\documentclass[a4paper,12pt]{article}
\usepackage{amssymb,amsthm,amsmath,amsfonts}
\usepackage{graphicx}
\usepackage[left=1.5cm,right=1.5cm,top=2cm,bottom=2cm]{geometry}
\usepackage{cite}
\usepackage{color}
 \usepackage{mathrsfs}

\newtheorem{theorem}{Theorem}[section]
\newtheorem{lemma}[theorem]{Lemma}
\newtheorem{corollary}[theorem]{Corollary}
\newtheorem{proposition}[theorem]{Proposition}
\newtheorem{definition}[theorem]{Definition}
\newtheorem{example}[theorem]{Example}

\theoremstyle{remark}
\newtheorem{remark}{Remark}

\newcommand{\bR}{{\mathbb{R}}}

\newcommand{\cA}{{\mathcal{A}}}

\newcommand{\cQ}{\mathcal{Q}}
\newcommand{\cq}{ {\cQ_p}}

\title{Möbius invariant weighted reduced quaternionic $\cq$-modules}
\small{\author{Isidro Paulino-Basurto$^{(1)}$, Jos\'e Oscar Gonz\'alez-Cervantes$^{(1)}$,\\ Juan Bory-Reyes$^{(2)}$, Lino Feliciano Res{\'e}ndis-Ocampo$^{(3)}$} 
\vskip 1truecm
\date{\small $^{(1)}$ Departamento de Matem\'aticas, ESFM-Instituto Polit\'ecnico Nacional. 07338, Ciudad M\'exico, M\'exico\\ Email: jogc200678@gmail.com, isidropaba21@gmail.com\\ $^{(2)}$ {SEPI, ESIME-Zacatenco-Instituto Polit\'ecnico Nacional. 07338, Ciudad M\'exico, M\'exico}\\Email: juanboryreyes@yahoo.com\\ $^{(3)}${Departamento de Ciencias Básicas. División de Ciencias Básicas e Ingeniería. Unidad Azcapotzalco. Universdad Autonóma Metropolitana. 02128, Ciudad M\'exico, M\'exico\\Email: lfro@azc.uam.mx}}

\begin{document}

\maketitle

\begin{abstract}
In this paper, we introduce some weighted reduced quaternionic modules induced by a reduced quaternionic Cauchy-Riemann operator associated to arbitrary orthogonal bases of $\mathbb R^3$. A weighted conformal invariant property of the modules is given. 
\end{abstract}

\section{Introduction} 
Basic properties of the functional complex spaces $\cq$ are well known in the literature, namely, their invariance under  M\"obius transformations  of the unit disk. These can be found for instance in \cite{AlemanMas, PeAraFis, Peloso, PRT, Xiao} and the references quoted therein. 

Quaternionic analysis in $\mathbb R^3$ is described as a theory of quaternionic algebra-valued functions of three real variables based on   null-solutions of the so-called generalized $\psi$-Cauchy Riemann operator, where $\psi$ is a arbitrary ortonormal basis of $\mathbb R^3$, usually referred to as structural set. The theory of $\psi$-hyperholomorphic functions generalizes  the classical theory of complex holomorphic { function theory} and offers a refinement of classical harmonic analysis, see e.g. \cite{AB, MS, S, SV1, SV2}.

In recent years, an alternative approach to quaternionic analysis in $\bR^3$ explores how to adapt and extend the classical quaternionic analysis when using reduced quaternionic-valued functions. This approach leads to a broader and richer  mathematical framework with potential applications in various scientific and engineering disciplines, see \cite{BGL, GN1, GN2, GN3, GNL, MP, Ng, JSK, Pi}.

In \cite{AljuaidBakhit, SayedOmran, G, GKST, GM, LinoTovar, MART}, a definition and basic properties  of $\cq$-spaces for quaternion-valued functions of three or four real variables were introduced, using different scales of weighted spaces of functions, which generalize several properties of the usual complex spaces $\cq$. 

In \cite{GLS1, GLS2, GLS3, G} the authors introduced and used what they called conformal weights in order to establish isometric isomorphisms between two function spaces associated with conformally equivalent domains and  extending what is known as the conformal invariance property of holomorphic function spaces. 

The study of the invariance of higher-dimensional $\cq$-spaces under M\"obius transformations has also been addressed in the literature; see \cite{CWR, CD, CDGS}.

The paper is organized as  follows: Section \ref{Prel} presents some basic properties of the quaternionic ${}^{\psi}{\mathcal{Q}}_{p,h,\rho}$ spaces associated to a $\psi$-Cauchy-Riemann operator. The main results, such as a conformal covariance property of the $\psi$-Cauchy-Riemann operator  and a conformal invariant property of quaternionic ${}^{\psi}{\mathcal{Q}}_{p,h,\rho}$ spaces are presented in Section \ref{section3}. Finally, conclusion and future works are also presented at the end.
\section{Preliminaries}\label{Prel}

\subsection{A generalized  $\psi$-Cauchy-Riemann operator  in $\mathbb R^3$}
The standard basis of the skew-field of quaternions,  $\mathbb H$, is $\{{\bf e_0}, {\bf e_1}, {\bf e_2}, {\bf e_3}\}$,  where ${\bf e_0}$ is the unit and $\{{\bf e_1}, {\bf e_2}, {\bf e_3}\}$ are the imaginary units.  
Thus, every $q\in\mathbb H$ can be written as
$$q=x_0 {\bf e_0}+x_{1} {\bf e_1}+x_{2} {\bf e_2}+x_{3} {\bf e_3},$$ 
where $x_{k}\in \mathbb R$ for  $k= 0,1,2,3$. 
Addition of quaternions is defined componentwise
and its multiplication  obeys the following rules: 
$$ {\bf e_1}^{2}= {\bf e_2}^{2}= {\bf e_3}^{2}=-{\bf e_0},$$
$$ {\bf e_1}\, {\bf e_2}=- {\bf e_2}\, {\bf e_1}= {\bf e_3};\ \,   {\bf e_2}\, {\bf e_3}=- {\bf e_3}\, {\bf e_2}= {\bf e_1};\ \, {\bf e_3}\, {\bf e_1}=- {\bf e_1}\, {\bf e_3}= {\bf e_2}.$$ 
The  quaternionic conjugation of $q\in  \mathbb H$ is ${\overline q}:=x_0 {\bf e_0}-x_{1}{\bf e_1}-x_{2}{\bf e_2}-x_{3}{\bf e_3}$  and the norm  of $q$ is introduced as 
$$|q|  := \sqrt{ x^{2}_{0}+x^{2}_{1}+x^{2}_{2}+x^{2}_{3}}= \sqrt{ q\,\overline q}=\sqrt{{\overline q}\,q}.$$

Given  $u, v\in\mathbb H$ its scalar product is 
$$\langle u, v\rangle:=\frac{1}{2}(\bar u v + \bar v u).$$
An ordered set $ \{\psi_0=1, \psi_1, \psi_2, \psi_3 \}  \subset \mathbb H$ is called  structural set of $\mathbb H$  if it satisfies the orthonormality condition:  
$$ { {\overline{\psi_k} \psi_s+\overline{\psi_s}\psi_k}}=2 \langle \psi_k, \psi_s\rangle =2\delta_{k,s},$$
for $k,s=0,1,2,3$, where $\delta_{k,s}$ is the Kronecker's symbol. { {As an immediate consequence of this definition we have  the following equalities
 \begin{equation} \overline{\psi_s} =-\psi_s, \quad \psi_s^2 =-1\quad  \mbox{and}\quad   \psi_k\psi_s= -\psi_s\psi_k,\quad{ k\neq s}\quad \mbox{ for}\quad  k,\ s=1,\, 2,\, 3 . \label{eq1} 
\end{equation}}}
Denote $\overline{\psi}= \{\overline{\psi_0}, \overline{\psi_1}, \overline{\psi_2}\}$ and by $\psi_{std}:= \{{\bf e}_0, {\bf e}_1, {\bf e}_2  \} $ we mean the standard  structural set of $\mathbb R^3$. 

In our case we use the base $\psi= \{\psi_0, \psi_1, \psi_2\}$ of $\mathbb R^3$ and the real subspace of $\mathbb H$  span of $\{\psi_0, \psi_1, \psi_2\}$, which is denoted by ${}^{\psi}\mathcal A$. Note that ${}^{\psi}\mathcal A$ is not a sub-algebra of $\mathbb H$.

In addition, the real three-dimensional vector space $\mathbb R^3$ will be embedded in $\mathbb H$ by identifying the element $x= (x_0,x_1,x_2)\in \mathbb R^3$ with $x_0\psi_0 + x_1\psi_1+ x_2 \psi_2\in {}^{\psi}\mathcal A$.  

By $\textrm{Sc}(x)= x_0$ we mean the scalar part of $x$ and denote $\mathbb B(0,1)=\{x\in {}^{\psi}\cA \ \mid \ |x|<1\}$.

Let $\Omega\subset\mathbb R^3 \cong {}^{\psi}\mathcal A$ be an open bounded domain with smooth boundary $\partial \Omega$ and $\overline\Omega$ denote its closure. Every $\mathbb H$-valued functions $f$ defined on $\Omega$ has the form $f=\sum_{k=0}^3 f_k \psi_k$, with $f_k:\Omega\to \mathbb R$, for $ k= 0,1,2,3$. The function $f$ is said to be continuously differentiable if each component$ f_k$ has {  this property}. As usual, the corresponding space is denoted by $C^{1}(\Omega, \mathbb H)$.
  
We define the left and the right-$\psi$-Cauchy-Riemann operators acting on $f \in C^1(\Omega, \mathbb H )$, respectively, by   
$${}^{{\psi}}\!Df := \sum_{k=0}^2 \psi_k( \partial_k f), \quad {}^{{\psi}}\!D_r f :=  \sum_{k=0}^2 (\partial_k f) \psi_k,$$ 
where we are using the partial differentiation operator $\partial_k f :=\displaystyle \frac{\partial f}{\partial x_k}$ for all $k$. 

These operators decompose the three-dimensional Laplace operator $\Delta_{\mathbb R^3}$ according to the real components of $x$. 
$$ { {^\psi\!  D}} \circ  ^{\overline{\psi}}\!\!  D=^{\overline{\psi}}\!\! D \circ ^{{\psi}}\!\! D =\,^{{\psi}}\! D_r\circ ^{\overline{\psi}}\!\! D_r=^{\overline{\psi}}\!\! D_r \circ ^{{\psi}}\!\! D_r =\bigtriangleup_{\mathbb R^3}, $$
where {} 
$^{\overline{\psi}}\!\! \  D f := \displaystyle\sum_{k=0}^2 {\overline{ \psi_k}}( \partial_k f)$ {} and {} $^{\overline{\psi}}\!\! \ D_r f := \displaystyle \sum_{k=0}^2 (\partial_k f) {\overline{ \psi_k}} $.

A function $f$ is said to be left-$\psi$-hyperholomorphic, or $\psi$-hyperholomorphic for short, (right-$\psi$-hyperholomorphic) on $\Omega$ if ${}^{{\psi}}\!D[f]=0$ (${}^{{\psi}}\!D_r[f]=0$) on $\Omega$.

Given $h\in C^1(\Omega,\mathbb H)$ the $\mathbb H$-right module ${}^{\psi}\!\mathcal M^h(\Omega)$ and (the $\mathbb H$-left module ${}^{\psi}\!\mathcal M_r^h(\Omega)$) consists of $f\in  C^1(\Omega,\mathbb H)$ such that ${}^{{\psi}}\!D[hf]=0$ (${}^{{\psi}}\!D_r[f h]=0$) on $\Omega$.

Given functions $g:\Xi \subset \mathbb H \to \Omega$ and $h:\Omega\to \mathbb H$, define the composition and multiplication  operators, respectively by
$$ W_g[f]= f\circ  g\qquad\hbox{and}   \qquad {}^h  M [f] = h \cdot f, \quad \mbox{for every}\qquad  f\in {C^1(\Omega,\mathbb H)}.$$ 
Then $f\in {}^{\psi}\!\mathcal M^h(\Omega)$ if and only if ${}^h \! M [f]$ is $\psi$-hyperholomorphic  function on $\Omega$.
\subsection{Quaternionic ${}^\psi\!{\mathcal{Q}_{p,\rho}^h}$ spaces} 

In this subsection we are going to extend the $\cq$ spaces presented in  \cite{GKST,GM} to the function space ${}^{\psi}\mathcal M^h(\Omega)$.

Given $a\in \mathbb B(0,1)\subset {}^{\psi}\mathcal A$ the M\"obius transformation $$ \varphi _a(x) = (a - x)(1 - \bar a x  )^{-1},$$
preserves $\mathbb B(0,1)$, i.e., $\varphi _a(\mathbb B(0,1))= \mathbb B(0,1)$. In addition, 
$$ g(x,a)  =\frac{1}{4\pi} \left({ \frac{1}{|\varphi_a(x)|} -1}\right), $$
is  the modified fundamental solution of the Laplacian in ${}^{\psi}\cA\cong \mathbb R^3$ composed with $\varphi _a$. See \cite{GKST,GM} for the standard base.   Given { {$p>0$  }} denote
$$g^p(x,a)  =\frac{1}{4^p\pi^p} \left( {  \frac{1}{|\varphi_a(x)|} -1}\right)^p.$$

{Let $\rho:{\mathbb B(0,1)}\to \mathbb R^+$ be a measurable function  and  $f \in {}^{\psi}\!\mathcal M^h(\mathbb B(0,1))$}. Consider the semi-norms
{ 
 \begin{align*} 
 {}^{ \psi}\!\mathcal B^h_\rho(f) :=&  \sup_{x\in \mathbb B(0,1)} (1-|x|^2)^{\frac{3}{2}} |{}^{\overline{ \psi} }\!D  f (x)| \rho(x), 
 \\ 
 {}^{ \psi}\!\mathcal Q^h_{p,\rho}(f) :=&  \sup_{a\in \mathbb B(0,1)}   \int_{\mathbb B(0,1)}| {}^{\overline{ \psi} }\!D   f (x)|^2 g^p(x,a) \rho(x) dB_x.
 \end{align*}
}
\begin{definition}
The spatial (or three-dimensional)  $h$-$\rho$-Bloch space ${}^{\psi}\mathcal B^h_\rho$ is the $\mathbb H$-right  module of all $f \in  {}^{\psi}\!\mathcal M^h({\mathbb B(0,1)})$ 
{with finite seminorm, i.e. } ${}^{ \psi}\!B^h_\rho(f ) < \infty$.

The $\mathbb H$-right module of all  $f\in {}^{\psi}\mathcal M^h({\mathbb B(0,1)}),$ which satisfy ${}^{ \psi}\! \mathcal Q^h_{p,\rho}( f ) < \infty $, is called ${}^{\psi}\!\mathcal Q^h_{p,\rho}$-space.

The $\mathbb H$-right  module of all $f\in {}^{\psi}\!\mathcal M^h({\mathbb B(0,1)}),$ which satisfy the condition
$$\int_{\mathbb B(0,1)} |{}^{ \overline{ \psi }}\! D  f(x)|^2\rho(x)dB_x{ <\infty},$$
where $dB_x$ denotes the three-dimensional volume {element}, is called spatial (or three-dimensional) Dirichlet space ${}^{\psi}\mathcal D^h_\rho$.
\end{definition}

\begin{remark}\label{remarkbasic}
The function quaternionic modules ${}^{ \psi_{std}}\mathcal B^1_1 $, ${}^{ \psi_{std}}\mathcal Q^1_{p,1}$ and  ${}^{ \psi_{std}} \mathcal D^1_1$ are studied in 
\cite{GKST}, obtaining results such as 
\begin{enumerate}
\item  For $0\leq p < 3$ one sees that ${}^{ \psi_{std}}\mathcal Q^1_{p,1}\subset {}^{\psi_{std}}\mathcal B^1_1$. \item  If $f\in {}^{ \psi_{std}}\mathcal M(\mathbb B(0,1))$ then  the following 
conditions are equivalent:
\begin{enumerate}
\item  $f\in  {}^{\psi_{std}}\mathcal B^1_1 $.
\item  ${}^{\psi_{std}} \mathcal Q^1_p(f)<\infty$  for all  $2<p<3$.
\item  ${}^{\psi_{std}} \mathcal Q^1_p(f)<\infty$ for some $p>2$.
 \end{enumerate}
\item Given $f\in {}^{ \psi_{std}}\mathcal M(\mathbb B(0,1))$ and  $ 1\leq p < 2.99$. Then $f \in {}^{ \psi_{std}}\mathcal Q^1_{p,1}$ if and only if 
$$\sup_{x\in \mathbb B(0,1)} \int_{\mathbb B(0,1)} | \overline{ {}^{ \psi_{std}} D} f(x)|^2(1- |\varphi_a (x)|^2)^p dB_x <\infty.$$
\item Given $f\in {}^{ \psi_{std}}\mathcal M(\mathbb B(0,1))$ and  $ 0< p \leq 1$. Then $f \in {}^{ \psi_{std}}\mathcal Q^1_{p,1}$ if and only if 
$$\sup_{{a}}\in \mathbb B(0,1) \int_{\mathbb B(0,1)} | \overline{ {}^{ \psi_{std}} D} f(x)|^2(1- |\varphi_a (x)|^2)^p dB_x <\infty.$$
\item If  $0 <p < q$ then  ${}^{ \psi_{std}}\mathcal Q^1_{ { { p}},1} \subset {}^{ \psi_{std}}\mathcal Q^1_{{ { q}},1} $.
\end{enumerate}
\end{remark}
\section{Main results}\label{section3}
\subsection{Conformal property of ${}^{{\psi}}\cA$}

\begin{definition}
The basic  M\"obius transformations $T_i : {}^{\psi}\cA\to {}^{\psi}\cA$ for $i=1, \dots ,4$ are the following:
\begin{enumerate}   
\item Translation. $T_1(x): =   x+ u $, where $u\in {}^{\psi}\cA$ and $T_1(\infty)= \infty$.
\item Dilation.  $T_2(x) := \lambda x$, where $\lambda >0$ and $T_2(\infty)=\infty$.
\item Rotation $T_3(x) := r  x\psi_3 \overline{r} \, \overline{\psi_3} $, where $r\in \mathbb H$ with $\|r\|=1$  and $T_3(\infty)=\infty$. 
\item Inversion $T_4(x) :={{x^{-1}}}= \frac{\bar x}{|x|^2}$  for all $x\in\ {}^{\psi}\mathcal A \setminus\{0\}$,   $T_4(\infty)=0$  and $T_4(0)=\infty$.
\end{enumerate}
In addition, a M\"obius transformation  on ${}^{\psi}\mathcal A $ is
{a finite composition} of the four previous {elementary} transformations.  
\end{definition}

\begin{remark}
It is well-known that {every} quaternionic rotations {is of the form} $R_{r,s}(q)= rqs$ for all $q\in \mathbb H$, where $r,\ s\in \mathbb H$ {satisfy} $|r|=|s|=1$. In particular, the quaternionic rotations 
{{preserving}}  the real {{subspace}}  span of $\{\psi_1, \psi_2,\psi_3\}$ are {{precisely those of the form}} $R_{r,\bar r}$. 

Note that $R_{1,\psi_3}({}^{\psi}\mathcal A ) $ is   the real space span of $\{\psi_1, \psi_2,\psi_3\}$ and the inverse mapping  of $R_{1,\psi_3}$  is  $R_{1,\overline{\psi_3}}$. Therefore,    any quaternionic rotation that preserve ${}^{\psi}\mathcal A$ is given by $R_{1,{ {\overline{ \psi_3}}}} \circ R_{r,\bar r} \circ R_{1,\psi_3}$.
\end{remark}

\begin{proposition}\label{Moebius} {{A mapping}}
$T$ is a M\"obius transformation on ${}^{\psi}\mathcal A$ if and only if {{it admits the representation}}
\begin{align}\label{MoebiusR3} 
T(x)= (ax+b)(cx+d)^{-1}, x\in {}^{\psi}\mathcal A, 
\end{align}  
and $a,b,c,d\in\mathbb H$ satisfy:
\begin{enumerate}
    \item If $c=0$ then  $bd^{-1}  \in {}^{\psi}\cA$  and $$\dfrac{d^{-1}}{|d^{-1}|}  =  \psi_3 \dfrac{\overline{a}}{|a|} \overline{\psi _3}. $$
    \item If $c\neq 0$ then $ac^{-1}, d(b - ac^{-1}d)^{-1} \in {}^{\psi}\cA$
    and
     \[
    \dfrac{(b  -  ac^{-1}d)^{-1}}{|(b  -  ac^{-1}d)^{-1}|}  =  \psi_3 \dfrac{\overline{c}}{|c|}\overline{\psi_3} . 
    \]
\end{enumerate}
In addition,  if $c=0$ then  $T(\infty) = \infty$. But if $c\neq 0$ then 
$T( - c^{-1}d) =\infty$ and $T( \infty)= ac^{-1}$
\end{proposition}

\begin{proof} 
To facilitate the computations, we {{represent every quaternion with respect to the basis}}
$\{\psi_0, \psi_1, \psi_2, \psi_3 \}$ and  the proof is a consequence of the following identities:
\begin{enumerate}
\item If $c=0$ then
\[
T(x)  =    \dfrac{a}{|a|} |a|\, |d^{-1}| x \dfrac{d^{-1}}{|d^{-1}|} + bd^{-1} .
\]
Denoting 
\[
T_2 (x)  =  |a| |d^{-1}|x  , \qquad
T_3 (x)  =  \dfrac{a}{|a|}x\psi_3 \dfrac{\overline{a}}{|a|} \overline{\psi_3} , \qquad
T_1 (x)  =  x + bd^{-1}  ,
\]
we obtain 
\begin{equation} \label{TransMob1}
    T   =  T_1 \circ T_3 \circ T_2 .
\end{equation}
 
 \item If $c\neq 0$ then  
 \begin{align*} 
T(x) =&   ac^{-1} + v(cx+d)^{-1}  =  ac^{-1}  +  (cxv^{-1} + dv^{-1})^{-1}
\\
=&   \left(   \dfrac{c}{|c|}|c|\, |v^{-1}| x \dfrac{v^{-1}}{|v^{-1}|} + dv^{-1}  \right)^{-1}   + ac^{-1}  ,
\end{align*} 
where $v:= b  -  ac^{-1}d$. Therefore,
\begin{equation} \label{TransMob2}
    T   =    T_1\circ T_4 \circ S_1 \circ T_3 \circ T_2  ,
\end{equation}
where
\begin{align*}
& T_2(x)=   |c|\, |v^{-1}| x,   
\quad 
T_3(x)=     \dfrac{c}{|c|} x \dfrac{v^{-1}}{|v^{-1}|},   
\quad 
S_1(x)=      x +dv^{-1},  \\
&T_4(x)=  x^{-1},  
\quad T_1(x)=     x + ac^{-1} .  
\end{align*} 
\end{enumerate}
\end{proof}

\begin{definition}
Two domains $\Omega, \Xi \subset {}^{\psi}\cA \cong \mathbb R^3$  are {{conformally}} equivalent domains, if there exists a M\"obius transformation $T$ such that  $T(\Omega)=\Xi$.
\end{definition}

The following result, proof of which is immediate from \eqref{eq1}, is useful to prove the conformal covariant property of ${}^{\psi}D$.
\begin{lemma}\label{lemmaeq1}
  Let $r\in \mathbb{H}$ and $q=  q_0\psi_0  +q_1 \psi_1 +q_2 \psi_2 +q_3\psi _3$. Then
\begin{itemize}
  \item[\rm a)] $\overline{r} \overline{\psi_3}\,  \overline{\psi_j}\,  \overline{\psi_3}=-\overline{r} \overline{\psi_j}$ for $j=0,\, 1,\, 2$.
  \item[\rm b)]  \[
\sum_{i=0}^2 \psi_i q \psi_i=  -q_0 \psi_0 +q_1 \psi_1 +q_2 \psi_2 +3q_3 \psi_3= -\overline{q} +2q_3 \psi_3=
-\overline{q} +2\langle q, \psi_3\rangle \psi_3 .\]
    \end{itemize}
\end{lemma}

\begin{proposition}\label{conformalproperyofD}(Conformal covariant property of ${}^{\psi}D$)
Let $\Omega, \Xi \subset {}^{\psi}\cA \cong \mathbb R^3$ be   two conformal equivalent domains and  let $T$ be  a Möbius transformation   given by 
 \eqref{MoebiusR3} such that $T(\Omega) =\Xi$.
 Define 
 \begin{align*}
 A_T(x):= &\left\{  \begin{array}{ll}  \psi_3 \dfrac{\overline{a}}{|a|}\overline{ \psi_3} , & \ \ \textrm{if} \ \ c=0, \\  
  \psi_3 \dfrac{\overline{c}}{|c|}\overline{\psi_3}  \dfrac{\overline{cxv^{-1} + dv^{-1}}}{|cxv^{-1} + dv^{-1}|^3} , & \ \ \textrm{if} \ \ c\neq 0,  \end{array} \right. \\
B_T(x):= &\left\{  \begin{array}{ll} \ {{|d^{-1}|\overline{a}}}  , & \ \ \textrm{if} \ \ c=0 ,\\
 -|v^{-1}| \overline{c} \dfrac{cxv^{-1} + dv^{-1}}{|cxv^{-1} + dv^{-1}|^5} , & \ \ \textrm{if} \ \ c\neq 0,  \end{array} \right.
  \end{align*}
 where $v:= b  -  ac^{-1}d$. Then 
  \begin{align}
  \label{invarianCauchy}
 {}^{\psi}\!D[A_{T} f\circ T] = 
 B_{T}  \, {}^{\psi}\!D [f]\circ T, \quad  \textrm{on } \ \ \Omega, 
  \end{align}
    for  all $f\in C^{1}(\Xi,\mathbb H)$. 
\end{proposition}
\begin{proof}
The basic Möbius transformations satisfy the following: 

\noindent
Translation.  If $y= T(x) = x + u$, where $u\in{}^{\psi}\cA$, then 
\begin{align*}
{}^{\psi}\!D [f\circ T](x)   
= &       \sum_{k=0}^{2}\psi_k \sum_{j=0}^{2}\dfrac{\partial f}{\partial y_j} (T(x) )\dfrac{\partial y_j}{\partial x_k} (x)  \\
= &     \sum_{j=0}^{2} \psi_j\dfrac{\partial f}{\partial y_j} (T(x) ) = {}^{\psi}D [f]\circ T(x)  ,  
 \end{align*}
where $y_j = x_j + u_j$ for  $j=0,1,2$.

\noindent
 Dilation. If $y= T(x)  =  \lambda x$, where $\lambda >0$, then
\begin{align*}
{}^{\psi}\!D [f\circ T](x)   
= &          \lambda \sum_{j=0}^{2} \psi_j\dfrac{\partial f}{\partial y_j} (T(x) ) =   \lambda  {}^{\psi}\!D [f]\circ T(x),  
\end{align*}
where $y_j = \lambda x_j $ for  $j=0,1,2$.

\noindent
 Rotation.  $y  =  T(x)  =  rx \psi_3 \overline{r} \, \overline{\psi_3}$, where 
$r\in \mathbb H$ and $|r|=1$. Note that 
\[
y_j  =  \langle   rx\psi_3 \overline{r} \, \overline{\psi_3} , \psi_j    \rangle    =     \dfrac{1}{2}(\overline{rx\psi_3 \overline{r} \, \overline{\psi_3}}\psi_j + \overline{\psi}_j rx\psi_3 \overline{r} \, \overline{\psi_3})
\] 
{ {
Since  ${{x= x_0 \psi_0 +x_1 \psi_1 +x_2\psi_2 \in  {}^{\psi}\cA }}$, we have 
\[\dfrac{\partial y_j}{\partial x_i} = \dfrac{1}{2} \left( \overline{r\psi_i\psi_3 \overline{r} \overline{\psi_3}} \psi_j +\overline{\psi_j}r\psi_i\psi_3 \overline{r} \overline{\psi_3}\right)\]
and 
\begin{align*}
 \psi_i \psi_3 \overline{r} \overline{\psi_3} \dfrac{\partial y_j}{\partial x_i} 
&=\dfrac{1}{2}  \psi_i \psi_3 \overline{r} \overline{\psi_3}
\left( \psi_3 r \overline{\psi_3}\, \overline{\psi_i}\overline{r} \psi_j +\overline{\psi_j}r\psi_i\psi_3 \overline{r} \overline{\psi_3}\right)\\ 
&=\dfrac{1}{2} \left(\overline{r} \psi_j + \psi_i \big(\psi_3 \overline{r} \overline{\psi_3} \,\overline{\psi_j}r\big)\psi_i\psi_3 \overline{r} \overline{\psi_3}\right).
\end{align*}
}}
Thus
\begin{align*}
& 
{}^{\psi}D  [\psi_3 \overline{r} \, \overline{\psi_3} f\circ T] (x)   \\ 
=&    \sum_{i=0}^{2}  \psi_i \dfrac{\partial}{\partial x_i}  (\psi_3 \overline{r} \, \overline{\psi_3}  f\circ T ) (x)  
=    \sum_{i=0}^{2}  \psi_i \psi_3 \overline{r} \, \overline{\psi_3} \dfrac{\partial}{\partial x_i}    (f\circ T)  (x)
\\
=&     \sum_{i=0}^{2} \psi_i \psi_3 \overline{r} \, \overline{\psi_3}  \sum_{j=0}^{2} \dfrac{\partial f}{\partial y_j} (T(x)) \dfrac{\partial y_j}{\partial x_i} (x)\\
=&  \dfrac{1}{2} \sum_{i=0}^{2} \sum_{j=0}^{2} (\overline{r}\psi_j  +  \psi_i \psi_3 \overline{r}\,{ {\overline{\psi_3}\, \overline{\psi_j}}} r\psi_i \psi_3\overline{r}\, \overline{\psi_3})  \dfrac{\partial f}{\partial y_j} (T(x))
\\
=&  {{ \dfrac{3}{2} \overline{r}{}\,^{\psi}\!D [f]  (T(x)) }}  +   \dfrac{1}{2} \sum_{j=0}^2 \left[    \sum_{i=0}^{2} \psi_i (\psi_3 \overline{r}\,{ {\overline{\psi_3}\,\overline{\psi_j} }}r)\psi_i \right] \psi_3 \overline{r}\, \overline{\psi_3}    \dfrac{\partial f}{\partial y_j}  (T(x)).
\end{align*}
 
 Denote $q = \psi_3 \overline{r}\, { {\overline{\psi_3} \, \overline{\psi_j}}} r$ then  by Lemma \ref{lemmaeq1} (b)
\begin{align*}
\sum_{i=0}^2 \psi_i q \psi_i&=   
 -\overline{r}\psi_j \psi_3 r \overline{\psi_3}+\left(\overline{r}\psi_j \psi_3 r \overline{\psi_3}\psi_3  + \overline{\psi_3}\psi_3 \overline{r} \overline{\psi_3} \,\overline{\psi_j}r\right)\psi_3
\\ 
& = -\overline{r}\psi_j \psi_3 r \overline{\psi_3}+\overline{r}\psi_j \psi_3 r \psi_3  + \overline{r} \overline{\psi_3} \,\overline{\psi_j}r\psi_3
 \end{align*}
and by \eqref{eq1}
\begin{align*}
\sum_{i=0}^2 \left( \psi_i q\psi_i\right)\left(\psi_3 \overline{r} \overline{\psi_3}\right) &=
-\overline{r}\psi_j \psi_3 r \overline{\psi_3}\psi_3 \overline{r} \overline{\psi_3} +\overline{r}\psi_j \psi_3 r \psi_3\psi_3 \overline{r} \overline{\psi_3}  + \overline{r} \overline{\psi_3} \,\overline{\psi_j}r\psi_3
 \psi_3 \overline{r} \overline{\psi_3}\\ 
&= -\overline{r}\psi_j -\overline{r}\psi_j -\overline{r} \overline{\psi_3} \,\overline{\psi_j} \,\overline{\psi_3}.
\end{align*}

Therefore by Lemma \ref{lemmaeq1} (a)

\begin{align*}
 \dfrac{1}{2}\sum_{j=0}^2 \sum_{i=0}^2   \left( \psi_i q\psi_i\right)\left(\psi_3 \overline{r} \overline{\psi_3}\right)
\dfrac{\partial f}{\partial y_j}(T(x))&= \dfrac{1}{2} \sum_{j=0}^2 \left( -\overline{r}\psi_j -\overline{r}\psi_j -\overline{r} \overline{\psi_3} \,\overline{\psi_j} \,\overline{\psi_3}\right) \dfrac{\partial f}{\partial y_j}(T(x))\\ \\
&= -\overline{r} \sum_{j=0}^2 \psi_j  \dfrac{\partial f}{\partial y_j}(T(x))
+ \dfrac{\overline{r}}{2} \sum_{j=0}^2  \psi_j \dfrac{\partial f}{\partial y_j}(T(x))\\ \\
&=
-\dfrac{\overline{r}}{2}
\sum_{j=0}^2 \psi_j {{\dfrac{\partial f}{\partial y_j}}}[f](T(x)) \\ \\
&= -\dfrac{\overline{r}}{2} {{{}^{\psi}\!D [f]  (T(x))}} 
\end{align*}
and
\begin{align*}
&
{}^{\psi}\!D  [\psi_3 \overline{r} \, \overline{ \psi_3} f\circ T] (x)    
=   \overline{r} 
{}^{\psi}\!D  [f] (T(x)).   
  \end{align*}

\noindent

Inversion $y= T(x) = x^{-1}$, where $x\neq 0$.  
\begin{align*}
&{}^{\psi}D  \left[ \dfrac{\overline{x}}{|x|^3} f\circ T \right]  (x)\\
   & =  \sum_{k=0}^{2} \psi_k \left[  \left( \dfrac{\partial}{\partial x_k} \dfrac{\overline{x}}{|x|^3} \right) (f\circ T)(x)     +     \dfrac{\overline{x}}{|x|^3} \dfrac{\partial}{\partial x_k} (f\circ T)(x)   \right] \\
=& \left[  \dfrac{ |x|^3}{|x|^{6}} +  \sum_{k=1}^{2}   \dfrac{|x|^3}{|x|^{6}}  -   \sum_{k=0}^{2} \psi_k \dfrac{3x_k |x| \overline{x}}{|x|^{6}} \right](f \circ T)(x)
 +    \sum_{k=0}^{2} \psi_k \dfrac{\overline{x}}{|x|^3} \dfrac{\partial }{\partial x_k} (f\circ T) (x)
\\
=& \left[ \dfrac{3}{|x|^3}    -    \dfrac{3 |x| \overline{x}x}{|x|^6}   \right] (f \circ T)(x)    +         \sum_{k=0}^{2} \psi_k \dfrac{\overline{x}}{|x|^3} \dfrac{\partial }{\partial x_k} (f\circ T) (x)
\\
=&   \sum_{k=0}^{2} \psi_k \dfrac{\overline{x}}{|x|^3} \sum_{j=0}^{2} \dfrac{\partial f}{\partial y_j} (T(x)) \dfrac{\partial y_j}{\partial x_k} (x)
\end{align*}
In order to calculate the right hand side, we see that 
\[
y_j  =  \langle  \dfrac{\overline{x}}{|x|^2} , \psi_j  \rangle    =    \dfrac{1}{2} \left(    \dfrac{x}{|x|^2}\psi_j  +  \overline{\psi_j} \dfrac{\overline{x}}{|x|^2}   \right) .
\]
Therefore, 
\begin{align*}
&{}^{\psi}D  \left[ \dfrac{\overline{x}}{|x|^3} f\circ T \right]  (x) =    \sum_{j=0}^{2} \sum_{k=0}^{2}  \psi_k \dfrac{\overline{x}}{|x|^3}  \dfrac{\partial y_j}{\partial x_k} (x)   \dfrac{\partial f}{\partial y_j} (T(x))
\\  
=&  \sum_{j=0}^2   \left(     \psi_0 \delta_{j,0}  - \sum_{k=1}^2 \psi_k \delta_{j,k} \right) \dfrac{\overline{x}}{|x|^5} \dfrac{\partial f}{\partial y_j} (T(x)) 
 \\
 & - \sum_{j=0}^2  \left(    \sum_{k=0}^2 \psi_k x_k  \right) \dfrac{2\langle \overline{x} , 
 \psi_j \rangle \overline{x} }{|x|^7}\dfrac{\partial f}{\partial y_j} (T(x)) 
\\  
=&  \sum_{j=0}^2 \overline{\psi_j} \dfrac{\overline{x}}{|x|^5} \dfrac{\partial f}{\partial y_j} (T(x))  -   \sum_{j=0}^2 \dfrac{2\langle \overline{x} , \psi_j \rangle  }{|x|^5} \dfrac{\partial f}{\partial y_j} (T(x))
\\  
=& \sum_{j=0}^2 \left( \dfrac{\overline{\psi_j} \overline{x} - 2\langle \overline{x} , \psi_j \rangle  }{|x|^5} \right) \dfrac{\partial f}{\partial y_j} (T(x))
\end{align*}
As $  -x\psi_j  =  \overline{\psi_j} \overline{x} - 2 \langle \overline{x} , \psi_j \rangle $ for $j=0,1,2$, then 
\begin{align*}
{}^{\psi}D \left[ \dfrac{\overline{x}}{|x|^3} f\circ T \right]  (x)
=      \sum_{j=0}^2 \left( \dfrac{ -x\psi_j }{|x|^5} \right) \dfrac{\partial f}{\partial y_j} (T(x))   
=  -\dfrac{x}{|x|^5} \sum_{j=0}^2 \psi_j \dfrac{\partial f}{\partial y_j} (T(x)).
\end{align*}

\noindent
The expressions for $A_T$ and $B_T$ in the general case are obtained
from the preceding identities by using  the decompositions
\eqref{TransMob1} and \eqref{TransMob2} when $c=0$ and $c\neq0$,
respectively.
\end{proof}

\begin{remark}
The previous proposition complements the results presented in
\cite{GLS1,GLS2,GLS3}, {{where}} the conformal covariance properties of
the Fueter and Moisil--Theodorescu operators are used to establish
{{isometric isomorphisms operators}} between quaternionic Bergman modules.
\end{remark}
\begin{example} \normalfont  Verify the statement of the Proposition \ref{conformalproperyofD} for  the identity $f(x) = x_0\psi_0 +x_1\psi_1+x_2\psi_2 $  in ${}^{\psi}\cA$ and the inversion $T(x)= \dfrac{1}{x}=x^{-1}$.\end{example}
For 
$T(x)=x^{-1}$,
we have 
$
a=0,\ b=1,\ c=1,\ d=0
$. 
Hence
$
v=b-ac^{-1}d=1
$
and consequently
$
v^{-1}=1.
$
The expression for $A_T$ becomes
\[
A_T(x)
=
\psi_3\frac{\overline c}{|c|}\overline{\psi_3}
\frac{\overline{cxv^{-1}+dv^{-1}}}
{|cxv^{-1}+dv^{-1}|^3}
=
\frac{\overline x}{|x|^3}
\]

and  the expression for for $B_T$
is 
\[
B_T(x)
=
-|v^{-1}|\overline c
\frac{cxv^{-1}+dv^{-1}}
{|cxv^{-1}+dv^{-1}|^5}
=
-\frac{x}{|x|^5}.
\]
Thus
\[
{}^\psi\! D[f](x)
=
\sum_{j=0}^{2}\psi_j
\frac{\partial x}{\partial x_j}
=
\sum_{j=0}^{2}\psi_j^2=1-1-1=-1
\]
and 
\[B_T(x){}^\psi\! D[f](x)]= -\frac{x}{|x|^5}(-1)= \frac{x}{|x|^5}.\]

We now compute the left-hand side. Since
\[
f(T(x))=x^{-1}=\frac{\overline x}{|x|^2},
\]
we have
\[
A_T(x)f(T(x))
=
\frac{\overline x^{\,2}}{|x|^5}.
\]

By the product rule,
\begin{align*}
{}^\psi D\left[\frac{\overline{x}^{\,2}}{|x|^5}\right]
&=
\sum_{k=0}^2
\psi_k
\frac{\partial}{\partial x_k}
\left(\overline{x}^{\,2}|x|^{-5}\right)
\\
&=
\frac{1}{|x|^5}
\sum_{k=0}^2
\psi_k
\frac{\partial\overline{x}^{\,2}}{\partial x_k}
+
\sum_{k=0}^2
\psi_k\overline{x}^{\,2}
\frac{\partial}{\partial x_k}\frac{1}{|x|^5}\\
&= 
\frac{1}{|x|^5}
\sum_{k=0}^2
\psi_k
(\overline{\psi_k}\,\overline{x}
+
\overline{x}\,\overline{\psi_k})
-\frac{5}{|x|^7}
\sum_{k=0}^2
x_k\psi_k\overline{x}^{\,2}\\
&=
\frac{1}{|x|^5}
\left[
\sum_{k=0}^2
\psi_k\overline{\psi_k}\,\overline{x}
+
\sum_{k=0}^2
\psi_k\overline{x}\,\overline{\psi_k}
\right]
-
\frac{5x\overline{x}^{\,2}}{|x|^7}.
\end{align*}

We have 
\[
\sum_{k=0}^2
\psi_k\overline{\psi_k}\,\overline{x}
=
3\overline{x}.
\]
and 
\begin{align*}
\sum_{k=0}^2
\psi_k\overline{x}\,\overline{\psi_k}
&=
\overline{x}
-\psi_1\overline{x}\psi_1
-\psi_2\overline{x}\psi_2.
\end{align*}
By Lemma~\ref{lemmaeq1}, applied to
$q=\overline{x}\in{}^\psi\cA$, we have
\[
\sum_{k=0}^2
\psi_k\overline{x}\psi_k
=
-\overline{\overline{x}}
=
-x.
\]
Thus,
\[
\overline{x}
+\psi_1\overline{x}\psi_1
+\psi_2\overline{x}\psi_2
=
-x,
\]
which implies
\[
\psi_1\overline{x}\psi_1
+
\psi_2\overline{x}\psi_2
=
-x-\overline{x}.
\]
Consequently,
\begin{align*}
\sum_{k=0}^2
\psi_k\overline{x}\,\overline{\psi_k}=
\overline{x}
-\left(-x-\overline{x}\right)
=
x+2\overline{x}.
\end{align*}
Substituting the preceding identities, we obtain
\begin{align*}
{}^\psi D\left[\frac{\overline{x}^{\,2}}{|x|^5}\right]
&=
\frac{3\overline{x}+x+2\overline{x}}{|x|^5}
-
\frac{5x\overline{x}^{\,2}}{|x|^7}
=
\frac{x+5\overline{x}}{|x|^5}
-
\frac{5x\overline{x}^{\,2}}{|x|^7}=
\frac{x}{|x|^5}.
\end{align*}

\begin{corollary}\label{corollaryconformalproperyofD}
Let $\Omega, \Xi \subset \cA \cong \mathbb R^3$ be two conformal equivalent domains and let $T$ be a Möbius transformation given by \eqref{MoebiusR3} such that $T(\Omega) =\Xi$. Then $f\in {}^{\psi}\!\mathcal M^1(\Xi)$ if and only if $A_{T} f\circ T \in {}^{\psi}\!\mathcal M^1(\Omega)$; or equivalently, $f\in {}^{\psi}\!\mathcal M^1(\Xi)$ if and only if $f\circ T \in {}^{\psi}\!\mathcal M^{A_{T} }(\Omega)$. 
\end{corollary}

\begin{proposition}\label{inverOperta}
Let $\Omega, \Xi \subset \cA \cong \mathbb R^3$ be   two conformal equivalent domains, and  let $T$ be  a Möbius transformation given by \eqref{MoebiusR3} such that $T(\Omega) =\Xi$. Then 
\begin{align*}
A_T  (x)A_{T^{-1}} (T(x)) =  & A_{T^{-1}} (T(x)) A_T  (x) =1, \quad \mbox{for every $ x\in \Omega$}. 
\end{align*} 
\end{proposition}
\begin{proof}
{{The identities follow by considering}} the four basic M\"obius transformations. In particular, the proof of these identities for dilations and  translations are {{immediate}}. {{For a rotation}}
  $T(x) = r x \psi_3 \overline{r} \, \overline{\psi_3}$ we have $T^{-1}(y) = \bar r y \psi_3 r \overline{\psi_3}$ {{and}}
\begin{align*}
A_T  (x) = \psi_3 \overline{r} \, \overline{\psi_3}, & \quad A_{T^-1}  (y) = \psi_3   r  \overline{\psi_3}.
\end{align*}
If $T(x) =  x^{-1}$ then  $T^{-1}(y) =   y ^{-1}$,  by 
\begin{align*}
A_T  (x) = \frac{{{\overline{x}}}}{|x|^3} , & \quad A_{T^-1}  (y) = \frac{{{\overline{y}}}}{|y|^3}.  
\end{align*}
{{Therefore,
\begin{align*}
A_{T^{-1}}(T(x))
&=
\frac{\overline{x^{-1}}}{|x^{-1}|^3}
=
\frac{x/|x|^2}{1/|x|^3}
=
|x|x.
\end{align*}
It follows that
\begin{align*}
A_T(x)A_{T^{-1}}(T(x))
&=
\frac{\overline{x}}{|x|^3}|x|x
=
\frac{\overline{x}x}{|x|^2}
=1.
\end{align*}
Likewise,
\begin{align*}
A_{T^{-1}}(T(x))A_T(x)
&=
|x|x\frac{\overline{x}}{|x|^3}
=
\frac{x\overline{x}}{|x|^2}
=1.
\end{align*}
}}
Thus, the identities hold for each of the four basic M\"obius
transformations and, consequently, for every M\"obius transformation
on ${}^{\psi}\!\cA$.

\end{proof}

\subsection{Weighted quaternionic $\cq$-spaces}

From now on we consider $b\in \mathbb B(0,1)$ and 
$$T (x) = (- x +b)(- \bar b x +1  )^{-1}, \quad \mbox{ {for all}}\quad  x\in \mathbb B(0,1).$$
\begin{definition}
From Proposition \ref{conformalproperyofD} we see that 
\begin{align*}
    A_T (x)  & =  -\psi_3 \dfrac{b}{|b|}\overline{\psi_3} \dfrac{\overline{-\bar b x(b-(\bar b)^{-1})^{-1} + (b-(\bar b)^{-1})^{-1}}}{|-\bar bx(b-(\bar b)^{-1})^{-1} + (b-(\bar  b)^{-1})^{-1}|^3} ,
    \\
    B_T (x)    &= 2|b-(\bar b)^{-1}| b\dfrac{-\bar b x(b-(\bar b)^{-1} )^{-1} +  (b-(\bar b)^{-1} )^{-1} }{|-\bar bx(b-(\bar b)^{-1} )^{-1}  + (b-(\bar b)^{-1})^{-1}|^5}
   , \end{align*}
for all $x\in \mathbb B(0,1) $ and define 
\begin{align*}
h_T(x): =  & 1- 2\textrm{Sc}A_T(x) \cdot (\overline{A_T (x)})^{-1}, \\
 E_T(x): = & (B_T(x) )^{-1} \cdot \left(\frac{1-|T(x)|^2  } {1-|x|^2}\right)^{\frac{3}{2}}, \\
{{M_{T,a}(x)}}:=& \frac{  g^p(T(x),a) }{  g^p(x,a)} |B_T(x) |^{{{-2}}} {{|J_T(x)|}}, \\
N_T(x):= & |B_T(x) |^{-2} {{|J_T(x)|}},
\end{align*}
for all $x\in \mathbb B(0,1)$, where $\textrm{Sc} (A_T)$ is the scalar part of  $A_T$ and {{
$|J_T(x)|$ denotes the absolute value of the determinant of the
Jacobian matrix of $T$ at $x$.}}
\end{definition}

\begin{remark}
Note that according to Proposition \ref{conformalproperyofD} the identity  
\begin{align}
  \label{invarianCauchyconj}
 {}^{\overline{\psi}}D[\overline{A_T} f\circ T] = 
 \overline{B_T} {}^{\overline{\psi}}D [f]\circ T   
\end{align}
is valid  for all $f\in C^{1}( \mathbb B(0,1),\mathbb H)$. 
\end{remark}

\begin{proposition}\label{prop1}(Conformal invariant property of some ${}^{\psi}{\mathcal{Q}}_{p,h,\rho}$ spaces)
The quaternionic right linear operator $ {}^{ \overline{A_T}}M \circ W_ T $ has the following properties:
\begin{enumerate} 
\item  $ {}^{ \overline{A_T}}M \circ W_ T : {}^{\psi}\mathcal M^1   (\mathbb B(0,1)) \to  {}^{\psi}\mathcal M^{h_T}  (\mathbb B(0,1)) $. 
\item  $ {}^{ \overline{A_T}}M \circ W_ T \mid_{  {}^{\psi}\mathcal B^1_1 } {}^{\psi}\mathcal B^1_1 \to   {}^{\psi}\mathcal B^{h_T}_{E_T}   $. 
\item $ {}^{ \overline{A_T}}M \circ W_ T  \mid_ {{}^{\psi}\mathcal Q^1_{p,1} }: {}^{\psi}\mathcal Q^1_{p,1} \to   {}^{\psi}\mathcal Q^{h_T}_{p, M_T}   $. 
\item $ {}^{ \overline{A_T}}M \circ W_ T \mid _{   {}^{\psi}\mathcal D^h_\rho} : {}^{\psi}\mathcal D^1_1 \to   {}^{\psi}\mathcal D^{h_T}_{N_T}   $.
\end{enumerate}
\end{proposition}
\begin{proof} Denote $g = \overline{A_T} f\circ T$ and the proof  follows from the following identities:

\noindent
   1. {{Since
\[
2\operatorname{Sc}(A_T)=A_T+\overline{A_T},
\]
we have}}
\begin{align}\label{consecuen12}
 {}^{\psi}\! D [(1- 2\textrm{Sc}A_T \cdot (\overline{A_T})^{-1}  ) g] &=  
  {}^{\psi}\! D [(1- 2\textrm{Sc}A_T \cdot (\overline{A_T})^{-1}  ) \overline{A_T} f\circ T]
  \nonumber \\
  &=  {}^{\psi}\! D [(\overline{A_T} f\circ T- 2(\textrm{Sc}A_T) f\circ T ]   \nonumber \\ 
&  = -  {}^{\psi}\! D [  {A_T} f\circ T ] = - {B_T}{}^{\psi}\! D [   f] \circ T. 
 \end{align}
{{Thus, if
$f\in{}^{\psi}\mathcal M^1(\mathbb B(0,1))$, then
${}^{\psi}D[f]=0$, and consequently $
{}^{\psi}D[h_Tg]=0.
$
Hence
$
g\in{}^{\psi}\mathcal M^{h_T}(\mathbb B(0,1)).
$}}
 \noindent
   2.  
\begin{align*}
&{}^{ \psi } B^{h_T}_{E_T}( \overline{A_T} f\circ T  ) = \sup_{x\in \mathbb B(0,1)} (1-|x|^2)^{\frac{3}{2}} |\overline{{}^{\psi}D} \ [ \overline{A_T} f\circ T ] (x)| |E_T(x)|     \\ 
=&   \sup_{x\in \mathbb B(0,1)} (1-|x|^2)^{\frac{3}{2}} | {}^{\overline{\psi}}D [ \overline{A_T} f\circ T]  (x)| |E_T(x)|     \\  
= &  \sup_{x\in \mathbb B(0,1)}
 (1-|x|^2)^{\frac{3}{2}} | \overline{ B_T (x)} \  {}^{\overline \psi}D  [f] (T  (x) ) 
|
 |B_T(x) |^{-1}  \left(\frac{1-|T(x)|^2  } {1-|x|^2}\right)^{\frac{3}{2}}  \\ 
= &  \sup_{y\in \mathbb B(0,1)} (1-|y|^2)^{\frac{3}{2}} |  \overline{ {}^{ \psi}D } f  (y) |   
= {}^{ \psi } B^{1}_{1}(f ).
\end{align*}
where \eqref{invarianCauchyconj} was used. 
  
 \noindent
   3. {{  
Again, using \eqref{invarianCauchyconj},
\begin{align*}
&{}^{\psi}\mathcal Q^{h_T}_{p,M_{T}}
(\overline{A_T}(f\circ T))
\\
&=
\sup_{a\in\mathbb B(0,1)}
\int_{\mathbb B(0,1)}
\left|
{}^{\overline{\psi}}D
[\overline{A_T}(f\circ T)](x)
\right|^2
g^p(x,a)M_{T,a}(x)\,dB_x
\\
&=
\sup_{a\in\mathbb B(0,1)}
\int_{\mathbb B(0,1)}
|B_T(x)|^2
\left|
{}^{\overline{\psi}}D[f](T(x))
\right|^2
g^p(x,a)
\\
&\qquad\qquad\times
\frac{g^p(T(x),a)}{g^p(x,a)}
|B_T(x)|^{-2}|J_T(x)|\,dB_x
\\
&=
\sup_{a\in\mathbb B(0,1)}
\int_{\mathbb B(0,1)}
\left|
{}^{\overline{\psi}}D[f](T(x))
\right|^2
g^p(T(x),a)|J_T(x)|\,dB_x.
\end{align*}
Using the change of variables $y=T(x)$, we obtain
\[
{}^{\psi}\mathcal Q^{h_T}_{p,M_{T,a}}
(\overline{A_T}(f\circ T))
=
{}^{\psi}\mathcal Q^1_{p,1}(f).
\]
}}

\noindent
   4.{{Finally,
\begin{align*}
&\int_{\mathbb B(0,1)}
\left|
{}^{\overline{\psi}}\!D
[\overline{A_T}(f\circ T)](x)
\right|^2
N_T(x)\,dB_x
\\
&=
\int_{\mathbb B(0,1)}
|B_T(x)|^2
\left|
{}^{\overline{\psi}}\!D[f](T(x))
\right|^2
|B_T(x)|^{-2}|J_T(x)|\,dB_x
\\
&=
\int_{\mathbb B(0,1)}
\left|
{}^{\overline{\psi}}\!D[f](T(x))
\right|^2
|J_T(x)|\,dB_x
\\
&=
\int_{\mathbb B(0,1)}
\left|
{}^{\overline{\psi}}\!D[f](y)
\right|^2\,dB_y.
\end{align*}
This proves the four assertions.}} 
\end{proof}

\begin{corollary}\label{cor12}
${}^{ \overline{A_T}}M \circ W_ T : {}^{\psi}\mathcal M^1 (\mathbb B(0,1))   \to  {}^{\psi}\mathcal M^{h_T} (\mathbb B(0,1))  $ is a bijective operator and its inverse operator is  
${}^{ \overline{A_{T^{-1}}}}M \circ W_ {T^{-1}} : {}^{\psi}\mathcal M^{h_T}  (\mathbb B(0,1))   \to  {}^{\psi}\mathcal M^{1} (\mathbb B(0,1)) $  
\end{corollary}
\begin{proof}
Follows from Proposition \ref{inverOperta}.
\end{proof}
  
\begin{corollary}\label{cor123}\  {} 
\begin{enumerate} 
\item ${}^{\psi}\mathcal M^1   (\mathbb B(0,1)) $ and  $ {}^{\psi}\mathcal M^{h_T}  (\mathbb B(0,1)) $ are isomorphic $\mathbb H$-right modules. 
\item ${}^{\psi}\mathcal B^1_1 $ and ${}^{\psi}\mathcal B^{h_T}_{E_T}$ are homeomorphic and isomorphic $\mathbb H$-right  modules. 
\item ${}^{\psi}\mathcal Q^1_1 $ and ${}^{\psi}\mathcal Q^{h_T}_{p,M_T}$ are  homeomorphic and isomorphic $\mathbb H$-right   modules.  
\item ${}^{\psi}\mathcal D^1_1 $ and  $ {}^{\psi}\mathcal D^{h_T}_{N_T}$ 
are homeomorphic and isomorphic $\mathbb H$-right Hilbert modules.   
\end{enumerate}
\end{corollary}
From Remark \ref{remarkbasic}, Proposition \ref{prop1} and  Corollaries \ref{cor12} and \ref{cor123} we directly obtain the following result:

\begin{corollary} \label{cor1234}  {} \ 
\begin{enumerate}
\item For $0\leq p < 3$ one sees that   ${}^{ \psi_{std}}  \mathcal Q^{h_T}_{q,M_T}  \subset   {}^{ \psi_{std}}\mathcal B^{h_T}_{E_T} $. 
\item  If $f\in {}^{{{\psi_{std}}}}\mathcal M^{h_T}  (\mathbb B(0,1))$ then  the following conditions are equivalent:
\begin{enumerate}
\item  $f\in  {}^{\psi_{std}}  \mathcal B^{h_T}_{E_T}  $.
\item  ${}^{\psi_{std}}   \mathcal Q^{h_T}_{p, M_T}  (f)<\infty$  for all  $2<p<3$.
\item  ${}^{\psi_{std}}  \mathcal Q^{h_T}_{p, M_T}  (f)<\infty$   for some $p>2$.
\end{enumerate}
\end{enumerate}
\end{corollary}
  
\begin{remark} We should point out that results given in Proposition \ref{prop1} and Corollaries \ref{cor12}, \ref{cor123}, \ref{cor1234} enhance the study of quaternionic ${}^{\psi}{\mathcal{Q}}_{p,h,\rho}$ spaces appeared in \cite{GKST}.  
\end{remark}

\section*{Statements and Declarations}
\subsection*{Funding} This work was partially supported by Instituto Polit\'ecnico Nacional (grant numbers IND-2026-0491, IND-2026-0101) and SECIHTI (grant number 1179388).
\subsection*{Competing Interests} No competing interests appear to influence the work reported in this paper are disclosed by the authors
\subsection*{Author contributions} IPB and JOGC conceived the study and writing of the paper. JBR and LFRO oversaw the project progress. All authors provided critical feedback and helped shape the research, analysis and manuscript. 
\subsection*{ORCID}
\noindent
Isidro Paulino-Basurto: https://orcid.org/0009-0003-4550-4851
\\
Jos\'e Oscar Gonz\'alez-Cervantes: https://orcid.org/0000-0003-4835-5436
\\
Juan Bory-Reyes: https://orcid.org/0000-0002-7004-1794
\\
Lino Feliciano Res{\'e}ndis-Ocampo: https://orcid.org/0000-0002-3274-1677

\end{document}